\documentclass[12pt]{amsart}
\usepackage{amsmath,amssymb,fullpage}

\title[Stable and real-zero polynomials]
{Stable and real-zero polynomials in two variables}
\author[Grinshpan]{Anatolii Grinshpan}
\author[Kaliuzhnyi-Verbovetskyi]{Dmitry~S.~Kaliuzhnyi-Verbovetskyi}
\author[Vinnikov]{Victor Vinnikov}
\author[Woerdeman]{Hugo J.~Woerdeman}
\address{Department of Mathematics \\
Drexel University\\
3141 Chestnut St.\\
Philadelphia, PA, 19104}
\email{\{tolya,dmitryk,hugo\}@math.drexel.edu}
\address{Department of Mathematics\\
Ben-Gurion University of the Negev\\
Beer-Sheva, Israel, 84105} \email{vinnikov@math.bgu.ac.il}

\thanks{AG, DK-V, HW were partially supported by NSF grant DMS-0901628.
DK-V and VV were partially supported by BSF grant 2010432.}

\subjclass{15A15; 47A13, 13P15, 90C25} \keywords{Determinantal representation; multivariable polynomial;
(semi-)stable polynomial; stability radius; self-reversive polynomial; real-zero polynomial; Lax conjecture.}

\theoremstyle{plain}

\newtheorem{thm}{Theorem}[section]

\newtheorem{lem}[thm]{Lemma}

\newtheorem{prop}[thm]{Proposition}

\numberwithin{equation}{section}

\newcommand{\beq}{\begin{equation}}
\newcommand{\eeq}{\end{equation}}

\theoremstyle{remark}
\newtheorem{rem}[thm]{Remark}

\newtheorem{ex}[thm]{Example}

\newcommand{\diag}{\operatorname{diag}}

\newcommand{\spn}{\operatorname{span}}

\newcommand{\col}{\operatorname{col}}

\numberwithin{equation}{section}

\newcommand{\bbm}{\begin{bmatrix}}
\newcommand{\ebm}{\end{bmatrix}}
\renewcommand{\c}{\overline}

\begin{document}

\begin{abstract} For every bivariate polynomial $p(z_1, z_2)$ of bidegree $(n_1, n_2)$, with $p(0,0)=1$, which has no zeros  in the open unit bidisk, we construct a determinantal representation of the form 
$$p(z_1,z_2)=\det (I - K Z ),$$
where $Z$ is an $(n_1+n_2)\times(n_1+n_2)$ diagonal matrix with coordinate variables $z_1$, $z_2$ on the diagonal 
and $K$ is a contraction. We show that $K$ may be chosen to be unitary if and only if $p$ is a (unimodular) constant multiple of its reverse.
  
Furthermore, for every bivariate real-zero polynomial $p(x_1, x_2),$ with $p(0,0)=1$, we provide a construction 
to build a representation of the form $$p(x_1,x_2)=\det (I+x_1A_1+x_2A_2),$$ where $A_1$ and $A_2$ are Hermitian
matrices of size equal to the degree of $p$. 

A key component of both constructions is a stable factorization of a positive semidefinite matrix-valued polynomial 
in one variable, either on the circle (trigonometric polynomial) or on the real line (algebraic polynomial).
\end{abstract}

\maketitle

\section{Introduction}\label{sec:Intro}

Stability of multivariate polynomials is an important concept
arising in a variety of disciplines, such as Analysis, Electrical
Engineering, and Control Theory
\cite{Borcea,Wagner,BF,Kummert,Doyle,Gurvits}. In this paper, we
discuss two-variable polynomial stability with respect to the
open unit bidisk
$${\mathbb D}^2= \{(z_1, z_2) \in {\mathbb C}^2\colon|z_1|<1,\ |z_2|<1\}.$$
A bivariate polynomial will be called \textit{semi-stable} if it
has no zeros in $\mathbb D^2$, and \textit{stable} if it has no
zeros in the closure $\overline{\mathbb{D}}^2$. The
\textit{bidegree} of $p\in\mathbb{C}[z_1, z_2]$ is the pair $\deg
p=(\deg_1p, \deg_2p)$ of its partial degrees in each variable. The
\textit{reverse} of $p$ is defined as
$\overleftarrow{p}(z)=z^{\deg p}\bar{p}(1/z)$, where for
$z=(z_1,z_2)$ and $n=(n_1,n_2)$ we set $z^n=z_1^{n_1}z_2^{n_2}$,
$\bar{p}(z):=\overline{p(\bar{z})}$, $\bar{z}=(\bar{z}_1,
\bar{z}_2)$, and $1/z=(1/z_1, 1/z_2)$. A polynomial is
\textit{self-reversive} if it agrees with its reverse.\footnote{The
terminology adopted here is different from that in some other
sources. E.g., one can find in the literature ``stable" and
``strictly stable" corresponding to our ``semi-stable" and
``stable", ``dual" or ``inverse" corresponding to our ``reverse",
and ``unimodular" or ``self-inversive" corresponding to our
``self-reversive".} A semi-stable polynomial $p$ is
\textit{scattering Schur} \cite{BF} if $p$ and $\overleftarrow{p}$
are coprime, i.e., have no common factors.

For every semi-stable $p\in\mathbb{C}[z_1, z_2]$, with $p(0,
0)=1$, we construct a representation
\begin{equation}\label{eq:repr} p(z_1,z_2)=\det (I_{|n|} - K Z_n ),
\end{equation}
with $n=\deg p$ and $K$ a contraction, where $|n|=n_1+n_2$ and
$Z_n=z_1I_{n_1}\oplus z_2I_{n_2}$; see Theorem \ref{main}.
Although we follow a slightly different path to achieve this
result, we are essentially in the trail of Kummert
\cite{Kummert0,Kummert1, Kummert}, who established \eqref{eq:repr}
in the case of scattering Schur polynomials  \cite[Theorem
1]{Kummert}. Note that, given a contractive $K$, every polynomial
defined by \eqref{eq:repr} is semi-stable, so one gains practical
means of designing semi-stable bivariate polynomials. 

As an application of Theorem \ref{main}, we also establish in Theorem \ref{thm:unitary} a representation
\eqref{eq:repr} for semi-stable self-reversive polynomials, with $K$ a unitary
matrix; notice that semi-stable self-reversive polynomials are never scattering Schur.
This representation was previously established directly in \cite[Section 10]{GIK} in a somewhat different setting.

In the one-variable case, the situation is transparent: every
$p\in\mathbb{C}[z]$, with $p(0)=1$, can be written in
 the
form
$$p(z)=(1-a_1z)\cdots(1-a_nz)=\det \bigg(I_n-\bbm
a_1 &       &\\
    &\ddots &\\
    &       &a_n
\ebm\bbm
z &       &\\
    &\ddots &\\
    &       &z
\ebm\bigg),$$ where $1/a_i$ are the zeros of $p$ counting
multiplicities. Thus $p$ admits a representation \eqref{eq:repr},
with $K=\diag[a_1,\ldots,a_n]$. This representation is minimal in
size, $n=\deg p$, and in norm, $\|K\|=\max_{1\le i\le d}|a_i|$.
Observe that $p$ is semi-stable (respectively, stable) if and only
if $K=\diag[a_1,\ldots,a_n]$ is contractive (respectively,
strictly contractive); in particular, all zeros of $p$ are on the
unit circle if and only if \eqref{eq:repr} holds with $K$ a
diagonal unitary. We also note that $K$ can be chosen to have all,
with perhaps one exception, singular values equal to 1. By a
result of A. Horn \cite{AHorn}, this choice is realized by an
upper-triangular $K$ with eigenvalues $a_1, \ldots , a_n$ and
singular values equal to $1,\ldots , 1, |\prod_{i=1}^n a_i |$; see
also \cite[Theorem 3.6.6]{HornII}.

Our main result, Theorem \ref{main}, shows that in the
two-variable case we are as well able to find a representation
 \eqref{eq:repr} that is minimal both in the size and in norm of $K$. 
In particular, this means that for a two-variable 
semi-stable polynomial $p$ with $p(0,0)=1$, we can find a representation \eqref{eq:repr} with $n={\rm deg} p$ and $K$ a contraction. 

In three or more variables, such a result does not hold; indeed,
it follows from \cite[Example 5.1]{GKVW}
 that for $5/6<r<1$ the stable polynomial
$$q(z_1, z_2, z_3)=1+\frac r5\ z_1z_2z_3\bigg(z_1^2z_2^2+z_2^2z_3^2+z_3^2z_1^2-2z_1z_2z_3^2-2z_1z_2^2z_3-2z_1^2z_2z_3
\bigg), $$
does not have a representation \eqref{eq:repr} with $K$ a
contractive $9\times 9$ matrix, even though its degree is $(3,3,3)$. 
In general, the problem of finding
a representation \eqref{eq:repr} with some (not necessarily
contractive) matrix $K$ and $n=\deg p$ for a multivariable
polynomial $p$, is overdetermined (see \cite{GKVW}). It is also
unknown whether a semi-stable polynomial $p$ in more than two
variables admits a representation of the form \eqref{eq:repr} with
a contractive matrix $K$ of any size (see \cite{GKVW} for a
discussion). An alternative certificate for stability in any
number of variables is given in \cite{WLerer}.
The general problem of constructing linear determinantal representations of a polynomial is a well known
classical problem in algebraic geometry, see \cite{KV12} and the references therein.

The analog of semi-stable self-reversive polynomials for the real line
(as opposed to the unit circle in the complex plane) are {\em real-zero} polynomials
(or --- upon homegenization --- homogeneous {\em hyperbolic} polynomials first introduced by G\aa rding \cite{Ga51,Ga59}).
These polynomials and their determinantal representations were actively studied in recent years
in relation to semidefinite programming; 
we refer to \cite{Vppf} for a state of the art survey and further references.  
Using stable factorization of univariate matrix polynomials 
that are positive semidefinite on the real line, 
we construct in Section 4 a positive self-adjoint determinantal representation for two-variable real-zero polynomials,
i.e., a determinantal representation of the form $p(x_1,x_2)= p(0,0) 
\det(I+x_1A_1+x_2A_2)$ where $A_1$ and $A_2$ are complex
self-adjoint matrices of the size equal to the degree of $p$.
This reproves the main result of \cite{HV}, see also \cite{Ha}
(which amounts to the solution of the Lax conjecture for homogeneous
hyperbolic polynomials in three variables,
see \cite{LPR}), in a somewhat weaker form
(see also \cite[Section 5]{Vppf} and \cite{PV}). Indeed, in \cite{HV} it
was proven that $A_1$ and $A_2$ can be chosen to be real symmetric. The
advantage of the approach here is that the proof uses factorizations of
matrix polynomials (unlike the algebra-geometrical techniques used in \cite{HV}) 
making it especially suitable for computations.

\section{Norm-constrained determinantal representations}

Throughout the paper we will assume that the polynomials are non-constant. Although one can adjust the definitions to include the case of $p\equiv 1$,
it does not seem worth to do this. 
Given a non-constant 
bivariate polynomial $p$, its \emph{stability radius} is
defined as
$$s(p):= \max\Big\{ r>0\colon p(z)\ne0,\ z\in r{\mathbb D}^2\Big\}.$$
Thus $p$ is semi-stable if and only if $s(p)\ge 1$, and stable if
and only if $s(p)>1$.

\begin{thm}\label{main}
Let $p(z_1, z_2)$, with $p(0, 0)=1$, be a non-constant 
bivariate polynomial. Then $p$ admits a representation \eqref{eq:repr}
with $n = {\rm deg} p$ and  $\| K\| = s(p)^{-1}$. 
\end{thm}

 Before we delve into the bivariate case, let us
consider an alternative way to obtain \eqref{eq:repr} in the
univariate case that does not require us to compute the roots of
$p$. Let $p(z) = p_0 + \cdots + p_n z^n$ be a stable polynomial.
Then the classical matrix theory says that $p$ is the
characteristic polynomial of the associated companion matrix.
Writing this in a form especially useful for our purposes, we have
$$ p(z) = p_0 \det(I_n - z C_p), $$
where
\begin{equation}\label{Cp}  C_p = \bbm- \frac{p_1}{p_0} & 1 & & 0 \cr \vdots & & \ddots &\cr -\frac{p_{n-1}}{p_0} &
0 & & 1\cr - \frac{p_{n}}{p_0} & 0 & \cdots & 0 \ebm.
\end{equation} As $p$ is stable, all the eigenvalues of $C_p$ lie
in ${\mathbb D}$. Thus $C_p$ is similar to a strict contraction. For our purposes, it will suffice to find a
similarity to a (not necessarily strict) contraction.
To this end, we proceed by introducing the Bezoutian
$$ Q = AA^* - B^* B, $$ where
\begin{equation*} A =
\bbm p_0 & & \cr \vdots & \ddots & \cr p_{n-1} & \cdots & p_0
\ebm,\quad  B = \bbm p_{n} & \cdots & p_{1} \cr & \ddots & \vdots
\cr & & p_{n} \ebm.  \end{equation*}

The Schur--Cohn criterion  (see, e.g., \cite[Section 13.5]{LT}) tells us that $p$ is stable if and only if $Q>0$.
If we now factor $Q= PP^*$,
with $P$ a square (and thus automatically invertible) matrix,
then $K=P^{-1}C_pP$ is a contraction \cite{WLerer}.
To see this, one shows that $\bbm Q^{-1} &
C_p^*Q^{-1}\\Q^{-1} C_p &
Q^{-1}\ebm$ and consequently
$P(I-K^*K)P^*$ are positive semi-definite. Since the range of $K^*$ is contained in
the range of $P^*$, it follows that $\|K\|\le 1$.
We now have the
desired representation $p(z)=p_0 \det (I_n - z K )$.\\

This alternative derivation of a representation \eqref{eq:repr} in
a univariate case provides the basis for the construction of
\eqref{eq:repr} in the bivariate case, where now the coefficients
$p_i$ and the matrices $C_p$, $A$, $B$, $Q$, and $P$ will depend
on one of the variables.

We will need the following two lemmata.

\begin{lem}\label{prod} If the polynomials $p_1$ and $p_2$ admit the representation \eqref{eq:repr}
as in Theorem \ref{main}, then so does their product.
\end{lem}

\begin{proof} Although we formulate and
prove the statement for bivariate polynomials, it is obviously extended to any
number of variables.

Observe that if $n^{\prime}=\deg p_1$  and $n^{\prime\prime}=\deg p_2$,
then $\deg (p_1p_2) = n^{\prime}+n^{\prime\prime}$ and $s(p_1p_2)
= \min \{ s(p_1) , s(p_2) \}$. If $p_1=\det (I_{|n^{\prime}|}-K_1
Z_{n^{\prime}})$ and $p_2 = \det (I_{|n^{\prime\prime}|}-K_2
Z_{n^{\prime}} )$, then $p_1p_2=\det
(I_{|n^{\prime}+n^{\prime\prime}|}- \tilde{K} \tilde{Z})$, where
$\tilde{K}=K_1\oplus K_2$ and $\tilde{Z}=Z_{n^{\prime}}\oplus
Z_{n^{\prime\prime}}$. Applying a permutation  $T$ which
rearranges the coordinate variables on the diagonal of $Z$, we
obtain that $p_1p_2=\det (I_{|n^{\prime}+n^{\prime\prime}|}- K
Z_{n^{\prime}+n^{\prime\prime}})$ with $K=T\tilde{K}T^{-1}$ and
$Z_{n^{\prime}+n^{\prime\prime}}=T\tilde{Z}T^{-1}$. Since $T$ is
unitary, we obtain that
\begin{multline*}
\|K\|=\|\tilde{K}\|=\max\{\|K_1\|,\|K_2\|\}\\
=\max\{s(p_1)^{-1},s(p_2)^{-1}\}=( \min \{ s(p_1) , s(p_2)
\})^{-1}=s(p_1p_2)^{-1}.
\end{multline*}
\end{proof}
Recall that a polynomial $p$ is called \emph{irreducible} if it
has no nontrivial polynomial factors.
\begin{lem}\label{lem:dense} Stable irreducible polynomials of a fixed bidegree with the constant term $1$
are dense in the set of semi-stable polynomials of the same bidegree with the constant term $1$
\end{lem}
While we formulate and prove the statement here for bivariate polynomials, it can be obviously extended to any
number of variables.
We view polynomials of bidegree $(n_1, n_2)$ or less 
with constant term $1$ as points in the coefficient
 space $\mathbb C^N$, where $N=(n_1+1)(n_2+1)-1$.
Notice that if a sequence of polynomials in $\mathbb C^N$ converges to a polynomial of
 bidegree $(n_1, n_2)$, then the polynomials in the sequence will eventually have the same bidegree.
Notice also that stable polynomials form an open set in $\mathbb C^N$.
\begin{proof}[Proof of Lemma \ref{lem:dense}]
Observe that the set of reducible
polynomials of bidegree $(n_1, n_2)$ or less with constant term $1$ (the products of polynomials of
smaller bidegrees) 
is a finite union of images of polynomial mappings 
${\mathbb C}^{N'} \times {\mathbb C}^{N^{\prime\prime}} \to \mathbb C^N$ where $N' + N^{\prime\prime} < N$.
Hence irreducible polynomials
form an open dense subset of $\mathbb C^N$. Every semi-stable
polynomial $p$ is a limit of stable dilations
$p_r(z_1,z_2):=p(rz_1, rz_2)$ as $r \uparrow 1$. On the other
hand, we can approximate $p_r$ by irreducible polynomials in
$\mathbb{C}^N$. If an irreducible polynomial $q$ is sufficiently
close to $p_r$, then $q$ is also stable. We conclude that every
semi-stable polynomial in $\mathbb{C}^N$ can be approximated by
stable irreducible polynomials; moreover, a semi-stable polynomial
of bidegree $(n_1,n_2)$ can be approximated by stable irreducible
polynomials of the same bidegree.
\end{proof}

We will make use of the 1D system realization theory. A univariate
matrix-valued rational function $f$ is said to have a
(finite-dimensional) \textit{transfer-function realization} if
\begin{equation}\label{eq:tf}
f(z) = D + Cz (I-Az)^{-1}B
\end{equation}
for some complex matrix $\bbm A& B \cr C& D \ebm$. A realization
\eqref{eq:tf} of $f$ is called \textit{minimal} if the block $A$
is of minimal possible size. Every rational matrix-valued function
$f$ which is analytic and contractive on $\mathbb D$ has a minimal
realization \cite{KFA}; moreover, the system matrix
 $\bbm  A & B\\ C & D \ebm$ of a minimal realization of $f$ can be chosen to be contractive  \cite{Arov}.

\begin{proof}[Proof of Theorem \ref{main}]  Without loss of generality, we may assume that $s(p)=1$,
i.e., $p$ is semi-stable. Indeed, otherwise one proves the result
for the semi-stable $q$, where $q(z)=p(zs(p))$, resulting in a
contraction $K_q$, and then put $K=K_q/s(p)$, to get the desired
representation for
 $p$.

Next, if we can show the existence of a representation $\det
(I_{|n|}-KZ_n)$, with $K$ a contraction, for a dense subset of
semi-stable polynomials of bidegree $n$ with constant term $1$, then we are done.
Indeed, if $p^{(j)} = \det (I_{|n|}-K^{(j)} Z_n)$, with $K^{(j)}$
a contraction, and $p^{(j)} \to p$, then $p=\det (I_{|n|} - K Z_n
)$, where $K$ is a limit point of the sequence $\{ K^{(j)} \colon
j \in{\mathbb N} \}$ (which exists as the contractions in
${\mathbb C}^{|n|\times |n|}$ form a compact set). Thus, we are
allowed to make some generic assumptions on $p$. For starters, by
Lemma \ref{lem:dense}, we may assume that $p$ is stable and
irreducible.

We now start the proof of the existence of a representation \eqref{eq:repr} with $n=\deg p$ and $\| K \| \le 1$
for an irreducible stable polynomial $p$. Along the way, we make some other assumptions of genericity.

Expand $p$ in the powers of $z_2$, $p(z_1, z_2) = p_0( z_1 ) + \cdots + p_{n_2}( z_1) z_2^{n_2}$,
and introduce the companion matrix
\begin{equation}\label{Cjdef2}  \mathtt C (z_1) = \bbm- \frac{p_1(z_1)}{p_0(z_1)} & 1 & & 0 \cr \vdots &
& \ddots &\cr -\frac{p_{n_2-1}(z_1)}{p_0(z_1)} & 0 & & 1\cr
- \frac{p_{n_2}(z_1)}{p_0(z_1)} & 0 & \cdots & 0 \ebm, \end{equation}
and the triangular Toeplitz matrices
\begin{equation*} \mathtt A(z_1) =
\bbm p_0 (z_1)& & \cr \vdots & \ddots & \cr p_{n_2-1} (z_1) &
\cdots & p_0(z_1) \ebm,\quad \mathtt B(z_1) = \bbm p_{n_2} (z_1) &
\cdots & p_{1} (z_1) \cr & \ddots & \vdots \cr & & p_{n_2} (z_1)
\ebm.  \end{equation*} Form the Bezoutian $Q(z_1) := \mathtt
A(z_1) \mathtt A(1/\c z_1)^* - \mathtt B(1/\c z_1)^* \mathtt
B(z_1)$.

Since the polynomial $p(z_1, \cdot)$ is stable for every
$z_1\in\mathbb T$, we have that $Q(z_1)$ is positive definite for
every $z_1\in\mathbb T$ \cite[Section 13.5]{LT}. Then there exists
a $n_2\times n_2$ matrix-valued polynomial $P(z_1) = P_0 + \cdots
+ P_{n_1} z_1^{n_1}$ such that the factorization $Q(z_1) =
P(z_1)P(z_1)^*$, $z_1 \in {\mathbb T}$, holds and $P(z_1)$ is
invertible for every $z_1 \in\c{\mathbb D}$ \cite{Rosen, DR}.
Since $p_0(z_1)=p(z_1,0)\neq 0$ for every
$z_1\in\overline{\mathbb{D}}$, the rational matrix-valued function
$$ M(z_1) := P(z_1)^{-1}\mathtt C(z_1) P(z_1)$$ is analytic on
$\overline{\mathbb D}$. In fact, $M$ is also contractive there
\cite{WLerer}.
To see this, one shows that $\bbm Q(z_1)^{-1} &
\mathtt C(z_1)^*Q(z_1)^{-1}\\Q(z_1)^{-1}\mathtt C(z_1) &
Q(z_1)^{-1}\ebm$ and consequently
$P(z_1)(I-M(z_1)^*M(z_1))P(z_1)^*$ are positive semi-definite for
$z_1\in\mathbb T$. Since the range of $M(z_1)^*$ is contained in
the range of $P(z_1)^*$, it follows that $\|M(z_1)\|\le 1$ for
every $z_1\in\mathbb{T}$, and by the maximum principle, for every
$z_1\in\overline{\mathbb{D}}$.

{\em Claim:} Generically, the only poles of $M$ are the zeros of
$p_0$.

{\em Proof of claim}. We will first show that when $P$ has a zero
at $z_1=a$ of geometric multiplicity 1,  the corresponding vector
in the left kernel is a left eigenvector of $\mathtt C(a)$.
Indeed, first observe that by analytic continuation,
$Q(z_1)=P(z_1)P(1/\bar{z}_1)^*$, where the analyticity domains of
the rational matrix-valued functions on the two sides of the
equality coincide. Then the zeros of $P$ are exactly the zeros of
$Q$ that lie in ${\mathbb C} \setminus \overline{\mathbb D}$. Let
$z_1\neq 0$. Observe that
$$ R(z_1) := \begin{bmatrix} \mathtt A(z_1) & \mathtt B(1/\bar{z_1})^* \cr \mathtt B(z_1) & \mathtt A (1/\bar{z_1})^*
\end{bmatrix}$$
is the resultant for the polynomials $p(z_1, \cdot )$ and $g(\cdot )$, where
$$  g(z_2) =\overleftarrow{p} (z_1, z_2)/z_1^{n_1}=
 \overline{p_{n_1,n_2}}/z_1^{n_1} + \cdots + \overline{p_{00}} z_2^{n_2}. $$
 Thus
$\det R(z_1) =  0 $ if and only if $p(z_1,z_2) = 0
=\overleftarrow{p} (z_1, z_2)$ for some $z_2$ (in the terminology
of \cite{GW, GW2}: $(z_1, z_2)$ is an {\em intersecting zero}).
For such $z_1$ and $z_2$ we have that the row vector $v_{2n_2-1}
(z_2)$ is in the left kernel of $R(z_1)$; here
$$ v_k (z) = \begin{bmatrix} 1 & z & \cdots & z^k \end{bmatrix}  $$ When $\overline{p_0} (1/z_1) \neq 0$, we have that
$\mathtt A (1/\bar{z_1})^*$ is invertible, and
$$ \begin{bmatrix} v_{n_2-1} (z_2) & z_2^{n_2} v_{n_2-1} (z_2 ) \end{bmatrix}
  \begin{bmatrix} \mathtt A(z_1) & \mathtt B(1/\bar{z_1})^* \cr \mathtt B(z_1) & \mathtt A (1/\bar{z_1})^*
   \end{bmatrix} = 0 $$
implies
\begin{equation}\label{gen}  v_{n_2-1} (z_2) Q(z_1) = 0. \end{equation}
Notice that we used that $\mathtt A(w_1)^*$ and $\mathtt B(w_2)$
commute as they are both upper triangular Toeplitz matrices. Thus
when $P$ has a zero at $z_1=a$ of geometric multiplicity 1, its
vector in the left kernel is $v_{n_2-1}(z_2)$, where $(a,z_2)$ is
an intersecting zero\footnote{Alternatively, one may use the
formula
$$ \frac{ p(z_1 , z_2) \overline{p(1/\bar{z_1}, z_2 )} - 
\overleftarrow{p} (z_1 , z_2) \overline{\overleftarrow{p}
(1/\bar{z_1}, z_2 )}}{1-|z_2|^2} = v_{n_2-1}(z_2) Q(z_1) v_{n_2
-1} (z_2)^* $$ to come to the same conclusion. This formula can be
easily checked by hand, but also appears in many sources;
 see, e.g., \cite[Section 4]{KVM}.}. It is now straightforward to check that $v_{n_2-1}(z_2)$ is a
  left eigenvector of $\mathtt C(a)$ corresponding to the eigenvalue $1/z_2$.

To show that the only poles of $M$ are the zeros of $p_0$, we
observe that the only other possible source of poles of $M$ would
be the zeros of $P$. Assuming (a generic condition!) that such a
zero $a$ has multiplicity 1 and is not a zero of $p_0$, we obtain
that
$$ P(z_1) = P_0 + P_1 (z_1-a) + \cdots +P_{n_1} (z_1-a)^{n_1},\quad P(z_1)^{-1} = S_{-1}/(z_1-a) + S_0 + S_1 (z_1-a) +
\cdots, $$  where $\dim {\rm Ker} P_0 = 1$ and ${\rm rank} S_{-1}
= 1$ (see, e.g., \cite[Chapter II]{BGK}). In addition, as $P(z_1)$
and $P(z_1)^{-1}$ multiply to $I_{n_2}$, we have $P_0 S_{-1} = 0 =
S_{-1} P_0$. By the result of the previous paragraph, we must have
that $S_{-1}= w v_{n_2-1}(z_2)$, for some column vector $w$. But
then we have that
$$ M(z_1) = P(z_1)^{-1} \mathtt C(z_1) P(z_1) = S_{-1} \mathtt C(a) P_0/(z_1-a) + G(z_1), $$
where $G(z_1)$ is analytic in a neighborhood of $a$. Since
$$v_{n_2-1} (z_2) \mathtt C(a) P_0 = (1/z_2) v_{n_2-1} (z_2) P(a) = 0,$$
we have $ S_{-1} \mathtt C(a) P_0 = 0$, and thus $M(z_1)$ does not
have a pole at $a$. This proves the claim.

We assume now for $p$ the generic assumptions above and, in
addition, the assumption that $p_{n_2}$ is of degree $n_1$ and is
coprime with $p_0$, which is also generic. Then the McMillan
degree of $\mathtt C$, and hence, of $\mathtt M$, is $n_1$,
therefore there exists
 a minimal contractive realization of $M$ with $A$ a $n_1\times n_1$ matrix:
$$ M(z_1) = D + Cz_1 (I_{n_1}-Az_1)^{-1} B $$
(see \cite[Section 4.2]{BGK} or \cite[Sections 4.1, 4.2]{BGR} for the notion of the McMillan degree of a
rational matrix-valued function and its equality to the size of a
minimal realization of the function). We have
$$p(z_1, z_2)=p_0(z_1)\det(I_{n_2}-M(z_1)z_2)=\frac{p_0(z_1)}{\det(I_{n_1}-z_1A)}\
\det\left(I_{|n|}-\bbm A & B\\ C & D\ebm\bbm z_1I_{n_1}& 0\\ 0 &
z_2I_{n_2}\ebm\right).$$ As $\det(I_{n_1}-z_1A)$ is the
denominator of the coprime fraction representation of $\det M$
\cite[Section 4.2]{BGR}, it follows that $p_0(z_1)=\det(I_{n_1}-z_1A)$. This
proves that
$$p(z_1, z_2)= \det\left(I_{|n|}-\bbm A & B\\ C & D\ebm\bbm
z_1I_{n_1}& 0\\ 0 & z_2I_{n_2}\ebm\right),$$ and we are done.
\end{proof}

Notice that the proof outlines a procedure to find a representation \eqref{eq:repr}. Let us try this out on a simple 
example.

\begin{ex} Let $p(z_1,z_2) = 1+az_1+bz_2$, where $a+b<1$ and $a, b>0$. Then $p_0(z_1) = 1+az_1 $ and $p_1(z_1) = b$.
We have
$$ M(z_1) =\mathtt{C}(z_1)=- \frac b{1+a z_1} =- b+\sqrt{ab}\ z_1(1+az_1)^{-1}\sqrt{ab}$$
and obtain the representation $p(z_1, z_2) = \det \left( I_2 -
\bbm -a & \sqrt{ab} \cr \sqrt{ab} & -b \ebm \bbm z_1 & 0 \cr 0 &
z_2 \ebm \right)$ with a contraction $\bbm -a & \sqrt{ab} \cr
\sqrt{ab} & -b \ebm $.
\end{ex}

Kummert \cite{Kummert0,Kummert1,Kummert} proved Theorem \ref{main}
for bivariate scattering Schur polynomials. He first constructed
for such a polynomial $p$ a \emph{2D Givone--Roesser system
realization} \cite{GR} of $f=\overleftarrow{p}/p$:
\begin{equation}\label{eq:GR}
f=D+CZ_n(I_{|n|}-AZ_n)^{-1}B,
\end{equation}
with the complex $(|n|+1)\times(|n|+1)$ matrix $\begin{bmatrix} A
& B\\
C & D
\end{bmatrix}$ being unitary, and then wrote it as
$$\frac{\overleftarrow{p}}{p}=\frac{\det\begin{bmatrix}
I_{|n|}-AZ_n
& B\\
-CZ_n & D
\end{bmatrix}}{\det(I_{|n|}-AZ_n)}.$$ Since the fraction
representation on the left-hand side is coprime and the bidegree
of the polynomial in the denominator of the fraction on the
right-hand side is less than or equal $n=(n_1,n_2)$ (in the
componentwise sense), the denominators must be equal:
$$p=\det(I_{|n|}-AZ_n).$$
Since $A$ is a contraction, Theorem \ref{main} follows for
this case, with $K=A$. We also remark that in this construction $K$ has all singular values, except one, equal to 1.

Let us note that the existence of 2D Givone--Roesser unitary system realizations
 was proved by Agler \cite{Ag} for a much more general class of contractive analytic operator-valued
 functions on
 the bidisk $\mathbb{D}^2$, however the (unitary) system matrix in such a realization has, in general,
 infinite-dimensional Hilbert-space operator blocks (in particular, $A$ is a contraction on an infinite-dimensional
  Hilbert space).
 Kummert's result in the special case of scalar rational inner functions is sharper in the sense that it provides
 a concrete finite-dimensional unitary realization of the smallest possible size. We also remark that an
 alternative construction of a finite-dimensional  Givone--Roesser unitary system realization for
 matrix-valued rational inner functions of two-variables is given in \cite{BSV}.

 The general case of Theorem \ref{main} can also be deduced from the special case of scattering Schur polynomials, since
 the latter is a dense set in the space of all bivariate polynomials of bidegree $n=(n_1,n_2)$ 
 with the constant term $1$, and the
 approximation argument as in our proof of Theorem \ref{main} works. In view of Lemma \ref{lem:dense}, it suffices
 to prove the following proposition.
 \begin{prop}\label{prop:SS-dense}
Every  stable irreducible polynomial is scattering Schur.
 \end{prop}
\begin{proof}
We will prove the statement here for bivariate polynomials, but it can obviously
  be extended to any number of variables.

 Suppose $p$ is a stable irreducible bivariate polynomial and is not scattering Schur. Then it must divide
 $\overleftarrow{p}$.

 If $p$ is a nontrivial polynomial depending only on
 one of the variables, say $z_1$,
  then it has the form $p(z_1,z_2)=a(z_1-z_0)$, with $a\in\mathbb{C}\setminus\{0\}$ and
  $z_0\in\mathbb{C}\setminus\overline{\mathbb{D}}$, and
  $\overleftarrow{p}(z_1,z_2)=-\bar{a}\overline{z_0}(z_1-1/\overline{z_0})$, which is impossible.

   If $p$ depends on both $z_1$ and $z_2$,
    then it is possible to fix one of the variables,
    say $z_2=\lambda$, on the unit circle $\mathbb{T}$ so that $q(z_1)=p(z_1,\lambda)$ is a nontrivial polynomial in
    $z_1$. Then $q$ has no zeros in $\overline{\mathbb{D}}$ and, since the polynomial
    $q(z_1)$ divides $\overleftarrow{p}(z_1,\lambda)$, $q$
    can not have zeros in $\mathbb{C}\setminus\overline{\mathbb{D}}$, a contradiction to the Fundamental
    Theorem of Algebra.

    Thus, $p$ is scattering Schur.
\end{proof}

We note that another stability criterion for bivariate
polynomials was established in \cite[Theorem 1.1]{Knese2008}. Namely, it was shown that $p$ is stable
if and only if
\begin{equation}\label{Knese} |p(z_1,z_2)|^2-|\overleftarrow{p} (z_1 , z_2)|^2\ge c (1-|z_1|^2)
(1-|z_2|^2), \ z_1, z_2 \in {\mathbb D},   \end{equation} for some $c>0$. Moreover, when $p$ is stable,
 one may choose $$ c = 4\pi \left( \int_{0}^{2\pi} \int_{0}^{2\pi} \frac{1}{|p(e^{i\theta} , e^{i\psi})|^2}
  d\theta d\psi \right)^{-1}. $$
A $d$-variable generalization of \eqref{Knese} may be found in \cite[Theorem 5.1]{Bickel}.

\section{The case of self-reversive polynomials}
Given a semi-stable polynomial $p$, one has the
factorization
$p=us,$ where  $u$ is a semi-stable self-reversive polynomial and $s$ is a scattering Schur polynomial
\cite[Theorem 4]{BF}. In the case where $p$ is semi-stable and self-reversive, the factor $s$ is a constant. Our next theorem
specializes the result of Theorem \ref{main} to this case. We first
establish several equivalent conditions for a semi-stable polynomial to be self-reversive;
while we formulate and prove the next proposition for bivariate polynomials, it is clear that it extends
to any number of variables.

\begin{prop}\label{prop:sssr}
Let $p$ be a semi-stable bivariate polynomial of bidegree $n=(n_1,n_2)$ with
$p(0,0)=1$; then the following statements are
equivalent:
\begin{itemize}
\item[(i)] $p$ is self-reversive up to a unimodular constant;
\item[(ii)] the coefficient of $z_1^{n_1}z_2^{n_2}$ in $p$ is unimodular;
\item[(iii)] if $(z_1, z_2) \in {\mathbb T}^2$, the one-variable polynomial $t\mapsto p(tz_1, tz_2)$
has all its zeros on ${\mathbb T}$.
\end{itemize}
\end{prop}
\begin{proof}
(i)$\Rightarrow$(ii) is obvious, since
for a bivariate polynomial $p$ of bidegree $n=(n_1,n_2)$,
the free term of $\overleftarrow{p}$ equals the conjugate of
the coefficient of $z_1^{n_1}z_2^{n_2}$ in $p$.

(ii)$\Rightarrow$(iii) Let $p(z_1,z_2)=\sum_{j=0}^rp_j(z_1,z_2)$, $p_0(z_1,z_2)=1$, be
the expansion of $p$ in homogeneous polynomials. Then
\begin{equation} \label{eq:onevarpoly}
p_{(z_1,z_2)}(t):=p(tz_1,tz_2)=\sum_{j=0}^rp_j(z_1,z_2)t^j.
\end{equation}
If the coefficient $p_{n_1,n_2}$ of $z_1^{n_1}z_2^{n_2}$ is unimodular,
then $r=n_1+n_2$ and $p_r(z_1,z_2)=p_{n_1,n_2}z_1^{n_1}z_2^{n_2}$.
For $(z_1,z_2)\in\mathbb{T}^2$, we can write
$$
p_{(z_1,z_2)}(t) = (1-a_1(z_1,z_2)t) \cdots (1-a_r(z_1,z)t),
$$
where $1/a_1(z_1,z_2),\ldots,1/a_r(z_1,z_2)$ are the roots of $p_{(z_1,z_2)}$ counting multiplicities.
Because of semi-stability, $|a_i(z_1,z_2)| \geq 1$;
but $p_r(z_1,z_2) = (-1)^r a_1(z_1,z_2) \cdots a_r(z_1,z_2)$
is unimodular, hence $|a_i(z_1,z_2)| = 1$.

(iii)$\Rightarrow$(i) Let $p_{(z_1,z_2)}$ be as in \eqref{eq:onevarpoly}.
By the assumption, for every $(z_1,z_2)\in\mathbb{T}^2$ the
polynomial $p_{(z_1,z_2)}$ is self-reversive up to a unimodular
constant, hence $p_r(z_1,z_2)$ is either zero or unimodular. Since
the polynomial $p_r$ is nonzero, it is not identically zero on
$\mathbb{T}^2$ (e.g., by the uniqueness principle for bivariate
analytic functions). By continuity, $p_r(z_1,z_2)$ is unimodular,
and thus $\deg p_{(z_1,z_2)}=r$ for every
$(z_1,z_2)\in\mathbb{T}^2$. It follows (e.g., by Rudin's
characterization of rational inner functions \cite[Theorem
5.2.5]{Rudin}) that $p_r$ is a monomial:
$$p_r(z_1,z_2)=p_{m_1,m_2}z_1^{m_1}z_2^{m_2},$$ with
$(m_1,m_2)\le(n_1,n_2)$, $|m|=r$, and $|p_{m_1,m_2}|=1$.

Now, the fact that for $(z_1,z_2)\in\mathbb{T}^2$ the
polynomial $p_{(z_1,z_2)}$ is self-reversive up to a unimodular
constant implies that
$$
\overline{p_{r-j}(z_1,z_2)} = \overline{p_r(z_1,z_2)} p_j(z_1,z_2)
$$
for $(z_1,z_2)\in\mathbb{T}^2$ and $j=0,\ldots,r$,
and therefore by analytic continuation
$$
\overline{p_{r-j}\Big(\frac{1}{\bar z_1},\frac{1}{\bar z_2}\Big)} =
\overline{p_r\Big(\frac{1}{\bar z_1},\frac{1}{\bar z_2}\Big)} p_j(z_1,z_2)
$$
for all $z_1,z_2 \neq 0$ and $j=0,\ldots,r$.
It follows that
$$
\overline{p\Big(\frac{1}{\bar z_1},\frac{1}{\bar z_2}\Big)}
=  \overline{p_r\Big(\frac{1}{\bar z_1},\frac{1}{\bar z_2}\Big)} p(z_1,z_2),
$$
and finally, since $p_r(z_1,z_2)=p_{m_1,m_2}z_1^{m_1}z_2^{m_2}$,
$$
z_1^{m_1}z_2^{m_2} \overline{p\Big(\frac{1}{\bar z_1},\frac{1}{\bar z_2}\Big)} = \overline{p}_{m_1,m_2} p(z_1,z_2).
$$
Comparing the degrees of $z_1$ and $z_2$ we see that $n_1=m_1$, $n_2=m_2$,
and $p$ is self-reversive up to the unimodular constant $\overline{p}_{m_1,m_2}$.
\end{proof}

\begin{thm}\label{thm:unitary}
Let the bivariate polynomial $p$ of bidegree $n=(n_1,n_2)\neq (0,0)$, with
$p(0,0)=1$, be semi-stable. Then $p$ is self-reversive up to a unimodular constant
if and only if $p$ admits a representation \eqref{eq:repr} with $n=\deg p$ and  $K$ unitary.
\end{thm}
\begin{proof}
The proof in one direction is immediate.
If $p=\det(I_{|n|}-KZ_n)$, with $K$ unitary and $n=\deg p$, then
\begin{multline*}
\overleftarrow{p}(z)=z^{n}\det(I_{|n|}-\overline{K}Z_n^{-1})=z^{n}\det(I_{|n|}-Z_n^{-1}K^*)
=\det Z_n\det(I_{|n|}-Z_n^{-1}K^*)\\
=\det(Z_n-K^*)
=\det(-K^*)\det(I_{|n|}-KZ_n)=\alpha p(z),
\end{multline*}
with $\alpha=\det(-K^*)\in\mathbb{T}$.

Conversely, assume that $p$ is self-reversive up to a unimodular constant,
or equivalently (by Proposition \ref{prop:sssr}) that the coefficient of $z_1^{n_1}z_2^{n_2}$ in $p$ is unimodular.
By Theorem \ref{main}, $p$ has a
representation \eqref{eq:repr} with a contractive $K$. Observe
that the modulus of
the coefficient of $z_1^{n_1}z_2^{n_2}$ equals $|\det K |$, which in turn equals the product of the singular values
of $K$. As $|\det K |=1$,  all singular values of $K$ must be equal to 1, yielding that $K$ is unitary.
\end{proof}

We notice that the procedure outlined in the proof of Theorem \ref{main} to find a representation 
\eqref{eq:repr} does not work for self-reversive polynomials as the Bezoutian $Q$ is $0$, 
and a limiting process (as in the beginning of the proof of Theorem \ref{main}) is necessary.

The non-trivial direction of Theorem \ref{thm:unitary} was previously established directly in \cite{GIK}
(generalizing the determinantal representations considered in \cite{AM}). 
More precisely, \cite[Theorem 10.5]{GIK} establishes that 
a bivariate polynomial $p$ of bidegree $(n_1,n_2)$ with $p(0,0)=1$, 
having no irreducible factors of the form $\alpha z_1 + \beta$, $\alpha,\beta \in {\mathbb C}$, 
admits a representation \eqref{eq:repr} with $n=\deg p$ and $K$ unitary, provided that
$p(z_1,z_2)=0$ and $|z_1|=1$ imply $|z_2|=1$, and $p(t,0)$ is a stable polynomial in $t$ (in the notation of
\cite[Theorem 10.5]{GIK}, this is the special case $n_2=0$, so $n_1=n$). 
It suffices therefore to notice the following proposition.

\begin{prop} \label{prop:semistabselfrev_vs_gik}
Let $p$ be a bivariate polynomial of bidegree $(n_1,n_2)$ with $p(0,0)=1$,
having no irreducible factors of the form $\alpha z_1 + \beta$, $\alpha,\beta \in {\mathbb C}$.
Then the following are equivalent:
\begin{itemize}
\item[(i)] $p$ is semi-stable and self-reversive up to a unimodular constant;
\item[(ii)] $p(z_1,z_2)=0$ and $|z_1|=1$ imply $|z_2|=1$, and $p(t,0)$ is stable in $t$.  
\end{itemize}
\end{prop}

\begin{proof}
Both (i) and (ii) are inherited by the irreducible factors of $p$, so we may assume without loss
of generalty that $p$ is irreducible with $n_2>0$.
We denote by $X$ the desingularizing Riemann surface of the projective closure of the zero set $Z_p$ of $p$ in
${\mathbb C}^2$, and we abuse the notation by letting $z_1$ and $z_2$ denote both the coordinates in ${\mathbb C}^2$
and the corresponding meromorphic functions on $X$.

(i) $\implies$ (ii) 
The second condition in (ii) is obvious.
Assume by contradiction that $p(z^0)=0$, $z^0=(z^0_1,z^0_2)$, $|z^0_1|=1$, $|z^0_2| \neq 1$.
Since $p$ is self-reversive, we may assume that $|z^0_2|<1$. Let $\xi^0 \in X$ lie above $z^0$.
Since $z_1$ is a non-constant meromorphic function on $X$, it is an open mapping, hence there exists $\xi \in X$
near $\xi^0$ such that $|z_1(\xi)|<1$ and $|z_2(\xi)|<1$ contradicting the semi-stability of $p$. 
 
(ii) $\implies$ (i)
\cite[Lemma 10.6]{GIK} shows that $p$ is self-reversive up to a unimodular constant.
Since $p$ is self-reversive, $X$ is endowed with an anti-holomorphic involution $\tau$, and the coordinate functions
$z_1$ and $z_2$ on $X$ are unimodular meromorphic functions (i.e., $1/\overline{z_1(\tau(\xi))}=z_1(\xi)$ and
similarly for $z_2$). The condition
$p(z_1,z_2)=0$ and $|z_1|=1$ imply $|z_2|=1$
implies further that $X$ is of dividing type, i.e., $X = X_+ \cup X_{{\mathbb R}} \cup X_-$, where the union is disjoint
and the three subsets in the decomposition are the preimages under $z_1$ of ${\mathbb D}$, ${\mathbb T}$, and 
the exterior of the unit disc (including $\infty$), respectively. 
Now, $p(t,0)$ is stable in $t$
means simply that all the zeroes of $z_2$ on $X$ lie in $X_-$, therefore $z_2$ is an inverse of an inner function
on $X_+$. So: $z_1$ is inner and $1/z_2$ is inner on $X_+$, implying that $p$ is semi-stable.
\end{proof}

We sketch now the construction of \cite{GIK} as adapted to our case, which is both simpler and more feasible
from a computational viewpoint than the general situation considered there.
Suppose $p\in\mathbb{C}[z_1,z_2]$ is a semi-stable self-reversive polynomial with $p(0,0)=1$.  
By Lemma \ref{prod}, it suffices to assume that $p$ is irreducible and obtain a representation \eqref{eq:repr} with
$\deg p=(n_1,n_2)$ and $K$ unitary for this case. 
We assume that $n_2>0$, the case $n_2=0$ being trivial.
It follows \cite[Lemma 10.7]{GIK}
that the polynomial $\overleftarrow{\partial p/\partial z_2}$ is semi-stable and
$\frac{z_2\partial p/\partial z_2}{\overleftarrow{\partial
p/\partial z_2}}$ is a coprime fraction representation of a
rational inner function. Then it is well known (see
\cite{Kummert,GW,BSV,Knese2008}) that there exist bivariate
polynomials $A_1$, \ldots, $A_{n_1}$ of bidegree $(n_1-1,n_2)$ or
less, and bivariate polynomials $B_1$, \ldots, $B_{n_2}$ of
bidegree $(n_1,n_2-1)$ or less, not all of them equal 0, such that
\begin{multline}\label{eq:CD}
 \overleftarrow{\frac{\partial p}{\partial z_2}} (z_1,z_2)  \overline{\overleftarrow{\frac{\partial p}{\partial z_2}}
  (w_1,w_2)}- z_2 \overline{w_2}  \frac{\partial p}{\partial z_2} (z_1, z_2) \overline{ \frac{\partial p}{\partial z_2}
  (w_1, w_2) }\\
 = (1-z_1 \overline{w_1}) \sum_{i=1}^{n_1}A_i(z_1,z_2)  \overline{ A_i(w_1,w_2)} + (1-z_2 \overline{w_2})
  \sum_{j=1}^{n_2}B_j(z_1,z_2)  \overline{ B_j(w_1,w_2)}
\end{multline}
(this replaces a more general decomposition with ``negative squares'' provided by \cite[Theorem 10.1]{GIK});
furthermore, these polynomials can be found using semidefinite programming software.
It is straightforward to verify the identity
$n_2p=\frac{\overleftarrow{\partial p}}{\partial z_2}+z_2\frac{\partial p}{\partial z_2}$,
 which implies that the left-hand side of \eqref{eq:CD} is equal to
$$n_2^2p(z_1,z_2)\overline{p(w_1,w_2)}-n_2z_2\frac{\partial p}{\partial z_2}(z_1,z_2)
\overline{p(w_1,w_2)}-n_2p(z_1,z_2)\overline{w_2\frac{\partial p}{\partial z_2}(w_1,w_2)}.$$
When both $(z_1,z_2)$ and $(w_1,w_2)$ lie in the zero set $Z_p$ of the polynomial $p$,
this expression and, thus, the left-hand side of \eqref{eq:CD} are equal to 0. Then we
use the standard ``lurking isometry'' argument. We first rewrite the equality \eqref{eq:CD}
restricted to $Z_p\times Z_p$ as
$$\begin{bmatrix}
A(w_1,w_2)\\
B(w_1,w_2)
\end{bmatrix}^*\begin{bmatrix}
A(z_1,z_2)\\
B(z_1,z_2)
\end{bmatrix}=\begin{bmatrix}
w_1A(w_1,w_2)\\
w_2B(w_1,w_2)
\end{bmatrix}^*\begin{bmatrix}
z_1A(z_1,z_2)\\
z_2B(z_1,z_2)
\end{bmatrix},$$
where $A(z_1,z_2):=\col_{i=1,\ldots,n_1}[A_i(z_1,z_2)]$ and $B(z_1,z_2):=\col_{i=1,\ldots,n_2}[B_i(z_1,z_2)]$.
Then we observe that this identity uniquely determines an isometry $$T\colon\spn\left\{\begin{bmatrix}
z_1A(z_1,z_2)\\
z_2B(z_1,z_2)
\end{bmatrix}\colon (z_1,z_2)\in Z_p\right\}\to\spn\left\{\begin{bmatrix}
A(z_1,z_2)\\
B(z_1,z_2)
\end{bmatrix}\colon (z_1,z_2)\in Z_p\right\},$$
defined on generating vectors by
$$T\colon\begin{bmatrix}
z_1A(z_1,z_2)\\
z_2B(z_1,z_2)
\end{bmatrix}\mapsto\begin{bmatrix}
A(z_1,z_2)\\
B(z_1,z_2)
\end{bmatrix}$$
and then extended by linearity.
We shall see a posteriori that in fact the span of the vectors on the right-hand side is all of ${\mathbb C}^{|n|}$,
so that the isometry $T$ is a unitary mapping of $\mathbb{C}^{|n|}$ onto itself.
At any rate, $T$ can be extended to a unitary mapping $K$ of $\mathbb{C}^{|n|}$ onto itself, i.e.,
to a $|n| \times |n|$ unitary matrix.

The nonzero polynomial  $P=\begin{bmatrix}
A\\
B
\end{bmatrix}\in\mathbb{C}^{|n|}[z_1,z_2]$ does not vanish identically on $Z_p$.
Indeed, B\'ezout's theorem \cite[p. 112]{Fulton} says that two
bivariate polynomials with no common factors can have at most a
finite number of common zeros equal to the product of total
degrees of the polynomials. Therefore, if $P$ vanishes identically
on $Z_p$, then the irreducible polynomial $p$ should divide every
component of $P$, but since these components, $A_i$,
$i=1,\ldots,n_1$, and $B_j$, $j=1,\ldots,n_2$, are polynomials of
smaller bidegree than $p$, this is impossible. Moreover, the set
$Z_p\setminus Z_P$ is Zariski relatively open and dense in $Z_p$.
Since the polynomial
 $q=\det (I_{|n|}-KZ_n)$ vanishes on this set, it vanishes on $Z_p$ as well. Applying B\'ezout's
 theorem
 again, we see that $p$ divides $q$. Since
 $\deg q=\deg p=n$ and $p(0)=q(0)=1$, we must have $p=q$, i.e., $p$ has a representation \eqref{eq:repr} with
 $n=\deg p$ and $K$ unitary. This provides an alternative proof of the non-trivial direction in
 Theorem \ref{thm:unitary}.

We notice that the restriction of
$P=\begin{bmatrix}
A\\
B
\end{bmatrix}\in\mathbb{C}^{|n|}[z_1,z_2]$ to $Z_p$
is a section of the kernel bundle of the determinantal representation $I_{|n|}-KZ_n$ of the irreducible
polynomial $p$, see \cite{V89,KV12}.
It follows (essentially since such a section is generated by the columns of the adjoint matrix
$\operatorname{adj}(I_{|n|}-KZ_n)$) that the entries of the restriction of $P$ to $Z_p$ are linearly
independent, in other words there exists no nonzero $c \in {\mathbb C}^{1 \times |n|}$ such that
$c P(z_1,z_2) = 0$ for all $(z_1,z_2) \in Z_p$.
Therefore the span of $P(z_1,z_2)$, $(z_1,z_2) \in Z_p$, is all of ${\mathbb C}^{|n|}$,
so that the isometry $T=K$ is already a unitary mapping of $\mathbb{C}^{|n|}$ onto itself
and no extension is needed.

We illustrate this on the following example.

\begin{ex} Let $p(z_1, z_2) = 1-z_1z_2 -\frac{1}{2} z_1^2 -\frac{1}{2} z_2^2 + z_1^2z_2^2$, so that $\deg p=(2,2)$.
We compute
$$ \frac{\partial p}{\partial z_2} (z_1, z_2) = -z_1 -z_2 + 2z_1^2 z_2,\quad
 \overleftarrow{\frac{\partial p}{\partial z_2}} (z_1,z_2) = 2 -z_1z_2 -z_1^2 ,
$$ and find (using semidefinite programming software)
$$ A_1(z_1,z_2)= \sqrt{2} (1-z_1z_2),\quad A_2(z_1,z_2)= z_1-z_2,$$ $$ B_1(z_1,z_2)=\sqrt{2}(1-z_1^2),\quad
  B_2(z_1,z_2) =z_1+z_2-2z_1^2z_2, $$
so that \eqref{eq:CD} holds.
Taking the zeros $(0, \sqrt{2})$, $(\sqrt{2}, 0)$,
$(\frac{1}{2}, -1+\frac{3}{\sqrt{2}})$, $( -1+\frac{3}{\sqrt{2}}, \frac{1}{2})$, we find that the unitary $K=T$ is
the matrix
$$ K= \frac{1}{\sqrt{2}} \begin{bmatrix} 0& 1 & 0 & 1\cr 1 & 0 & -1 & 0 \cr 0 & -1 & 0 & 1 \cr 1 & 0 & 1 & 0
\end{bmatrix}. $$
One can easily check that $p(z_1, z_2) = \det (I_4 - K Z_{(2,2)}).$
\end{ex}

\section{Real-zero polynomials and self-adjoint determinantal representations}

We consider bivariate {\it real-zero polynomials}, 
which are polynomials $p\in\mathbb{R}[x_1,x_2]$ 
with the property that for every $(x_1,x_2) \in {\mathbb R}^2$ 
the one-variable polynomial $p_{(x_1,x_2)}(t):=p(tx_1,tx_2)$ has only real zeros. 
In \cite[Theorem 2.2]{HV} it was shown that every real-zero polynomial $p$ with $p(0,0)=1$ may be represented as 
\begin{equation}\label{RZP} 
p(x_1,x_2)=\det (I+ x_1A_1 + x_2A_2), \end{equation}
where $A_1,A_2 \in {\mathbb R}^{d\times d}$ are symmetric matrices and $d$ is the total degree of $p$;
in the homogeneous setting of {\em hyperbolic polynomials} this statement was known as the Lax conjecture,
see \cite{LPR}. We refer to \cite{Vppf} for a detailed survey and further references.
The proof in \cite{HV}
is based on the results of \cite{Vin93} and \cite{BV99}, see also \cite{Dub83},
and uses algebro-geometrical techniques --- the correspondence between (certain)
determinantal representations of an irreducible plane curve and line bundles on its desingularization,
together with a detailed analysis of the action of the complex conjugation on the Jacobian variety
and the theory of Riemann's theta function;
a new proof, using instead the theory of quadratic forms, has been discovered recently in \cite{Ha}.
A somewhat weaker statement --- namely, the existence of a representation  \eqref{RZP} 
where now $A_1,A_2 \in {\mathbb C}^{d\times d}$ are {\em Hermitian} matrices --- has been established
recently in \cite[Section 5]{Vppf} and \cite{PV}; these proofs are also algebro-geometrical but
avoid the transcendental machinery of Jacobian varieties and theta functions. 
In this section (Theorem \ref{RZ}), 
we provide a new proof (actually, two closely related proofs) of the existence of a representation  \eqref{RZP} 
with $A_1,A_2 \in {\mathbb C}^{d\times d}$ Hermitian matrices using factorizations of matrix valued polynomials. 
One advantage of our proof is 
that it provides a fairly constructive way to find such a representation. 
The most involved step is finding a stable factorization for a one-variable matrix polynomial 
that is positive semidefinite on the real line. 
As the latter can be implemented using any semidefinite programming package or 
a Riccati equation solver (see, e.g., \cite{HW} or \cite[Section 2.7]{BW}),
this construction can be easily implemented numerically, for instance in Matlab. 
For more on computational questions related to the construction of determinantal representations
of real-zero polynomials, 
see \cite{Hen10,PSV1,PSV2,LP}.
By a simple trick,
Theorem \ref{RZ} also implies the existence of a $2d \times 2d$ real symmetric representation for $p^2$ --- 
see Remark \ref{rem:p-square}.

\begin{thm}\label{RZ}
Let $p$ be a bivariate real-zero polynomial of total degree $d>0$ with $p(0,0)=1$. 
Then there exist $d\times d$ Hermitian matrices $A_1$ and $A_2$ so that \eqref{RZP} holds.
\end{thm}

We will need two lemmata. The first one is simply a restatement of one of the results of \cite{Nuij}
in the non-homogeneous setting.

\begin{lem}\label{aux} Let $p$ be a real-zero polynomial of total degree 
$d$ and with $p(0,0)=1$. 
For every $\epsilon > 0$ there exists a real-zero polynomial $q$ of total degree 
$d$ and with $q(0,0)=1$ 
such that each coefficient of $q$ is within $\epsilon$ distance of the corresponding coefficient of $p$, 
and for every $x_2\in{\mathbb R}$ the one-variable polynomial $\check q_{x_2}$ defined via
$$\check q_{x_2}(t):= t^d q(1/t, x_2/t)$$
has only simple real zeros.
\end{lem}

\begin{proof} Let ${\mathbf p}(x_0,x_1,x_2):= x_0^d p(x_1/x_0,x_2/x_0)$. 
Then ${\mathbf p}$ is a degree $d$ homogeneous polynomial in three variables that is 
{\em hyperbolic} with respect to $e=(1,0,0)$, 
which means that ${\mathbf p}(e)\neq 0$ and for every $(x_0,x_1,x_2)\in{\mathbb R}^3$ 
the one-variable polynomial $t \to {\mathbf p}(x_0-t,x_1,x_2)$ has only real zeroes. 
By a result of \cite{Nuij}, the polynomial ${\mathbf p}$ can be approximated 
arbitrarily close, in the sense of coefficients, by a degree $d$ homogeneous polynomial ${\mathbf q}$ 
which is {\em strictly hyperbolic} with respect to $e$; 
that is, ${\mathbf q}$ is hyperbolic with respect to $e$, 
and for every $(x_0,x_1,x_2)\in{\mathbb R}^3$ with $(x_1,x_2)\neq (0,0)$ 
the zeros of $t \to {\mathbf q}(x_0-t,x_1,x_2)$ are simple. 
But then $q(x_1,x_2):={\mathbf q}(1,x_1,x_2)/{\mathbf q}(1,0,0)$ has the desired property.
Notice that while a priori the total degree of $q$ is at most $d$, it will be actually equal to $d$ 
if we choose $\epsilon$ small enough.
\end{proof}

The following result is due to C. Hanselka \cite{Ha}. For the sake of completeness, we include a proof.
\begin{lem} \label{polynomial_implies_linear}
Let $M$ be a $d \times d$ matrix-valued polynomial in one variable with Hermitian coefficients,
and assume that the polynomial $\det(tI_d - M(s))$ has total degree at most $d$; then $M$ is linear
(i.e., $\deg M \le 1$).
\end{lem}
\begin{proof}
Let
$$
\det(tI_d - M(s)) = t^d + p_1(s)t^{d-1} + \cdots + p_d(s),
$$
where $p_j$ is a polynomial of degree at most $j$.
Assume that $M$ is a polynomial of degree $k$, and write $-M(s) = B_0 + \cdots + B_k s^k$. The sum
of $j \times j$ principal minors in $-M(s)$ is exactly $p_j(s)$; 
therefore the coefficient of $s^{kj}$ in $p_j(s)$ is the sum of $j \times j$ principal minors in $B_k$. 
But $\deg p_j \leq j$ for all $j$, 
hence if $k>1$ we conclude that the sum of $j \times j$ principal minors in $B_k$ is zero for all $j>0$. 
It follows that $B_k$ is nilpotent. Since $B_k$ is also Hermitian, it must be zero, a contradiction.
\end{proof}

We will present two closely related proofs of Theorem \ref{RZ}:
the first proof uses the Hermite matrix
(considered in the context of real-zero polynomials and determinantal representations
in \cite{Hen10} and in \cite{NPT}),
whereas the second proof uses intertwining polynomials and the Bezoutian
(considered in this context in \cite{Vppf} and in \cite{PV,KPV}).

\begin{proof}[First Proof of Theorem \ref{RZ}]
We first claim that if we can establish the existence of a required determinantal representation
for a dense subset of real-zero polynomials of total degree $d$ and with constant term $1$, then we are done\footnote{
This was previously noticed in \cite[Lemma 8]{Sp05} and \cite[Lemma 3.4]{PV}}.
Indeed, assume 
that we have real-zero polynomials $p^{(n)}$, $n\in{\mathbb N}$, 
of total degree $d$ with $p^{(n)}(0,0)=1$,  
so that the sequence $\{p^{(n)}\}_{n\in{\mathbb N}}$ converges to $p$
and so that
there exist Hermitian $d\times d$ matrices $A_1^{(n)}$ and $A_2^{(n)}$ with
$p^{(n)}(x_1,x_2)=\det(I_d + x_1 A_1^{(n)} + x_2 A_2^{(n)})$. Let $$\mu := \min \{ |t|\colon \ p(t,0)p(0,t) = 0 \}.$$ 
Clearly, $\mu>0$. 
Then for $n$ large enough the spectra of $A_1^{(n)}$ and $A_2^{(n)}$ lie in the interval $(-2\mu^{-1},2\mu^{-1})$. 
Since the spectral radius of an Hermitian matrix coincides with its operator $(2,2)$ norm, the matrices
$A_1^{(n)}$ and $A_2^{(n)}$ have norms bounded by $2\mu^{-1}$, 
and therefore the sequence $\{(A_1^{(n)}, A_2^{(n)}) \}_{n\in{\mathbb N}}$, has a limit point $(A_1,A_2)$. 
Then we get that $p(x_1,x_2)=\det(I_d + x_1A_1+x_2A_2)$, with Hermitian $d\times d$  matrices $A_1$ and $A_2$, as desired.

Given $p$, we introduce
$$\check p_{x_2}(t):= t^d p(1/t, x_2/t) = t^d + p_1(x_2) t^{d-1} + \cdots + p_d(x_2). $$ 
One easily observes that $\deg p_j \le j$, $j=1,\ldots, d$, 
and that for every $x_2\in{\mathbb R}$ the polynomial $\check p_{x_2}$ has only real zeros. 
Furthermore, we may assume by the previous paragraph and by Lemma \ref{aux}
that for every $x_2\in{\mathbb R}$ the polynomial $\check p_{x_2}$ has  only simple zeros.

Let $C(x_2)$ be the companion matrix
$$C(x_2) = \begin{bmatrix} 0 & \cdots  & 0 & -p_d(x_2) \cr 
1 &  & 0 & -p_{d-1}(x_2) \cr & \ddots &  & \vdots \cr 0 & & 1 & -p_1(x_2) \end{bmatrix} . $$
Then $$\check p_{x_2}(t) = \det (tI_d -C(x_2)).$$
Denote the zeros of $\check p_{x_2}$ by $\lambda_1(x_2), \ldots, \lambda_d(x_2)$, 
and let $s_j(x_2)$ be their $j$th Newton sum:
$$s_j(x_2) = \sum_{k=1}^d \lambda_k(x_2)^j,\quad j=0, 1,\ldots.$$
As is well known, $s_j(x_2)$ can be expressed in terms of $p_j(x_2)$, as follows
$$s_0(x_2)=d,\quad s_1(x_2)= -p_1(z),\quad 
s_j(x_2) = -jp_j(x_2)-\sum_{k=1}^{j-1} p_k(x_2) s_{j-k}(x_2),\quad j=2,\ldots,d. $$
Note that $s_j$ is a polynomial of degree $\le j$, $j=0,\ldots,d$.
We let $H(x_2)$ be the {\em Hermite matrix} of $\check p_{x_2}$, namely (see, e.g., \cite{KN36}) 
the Hankel matrix whose entries are the Newton sums of the zeros of $\check p_{x_2}$: 
$$H(x_2) = [s_{i+j}(x_2)]_{i,j=0,\ldots,d-1}.$$ 
Clearly,  $H$ is a matrix polynomial of degree at most $2d$. E.g., for $d=2$ we have
$$H(x_2) = \begin{bmatrix} 2 & -p_1(x_2)\cr -p_1(x_2) & p_1(x_2)^2-2 p_2(x_2) \end{bmatrix}.$$ 
Since all the zeros of $\check p_{x_2}$ are real and simple for real $x_2$, 
we have that $H(x_2) > 0$, $x_2 \in {\mathbb R}$. 
This is well known and it follows immediately from
\begin{equation} \label{Vander_fact}
H(x_2) = V(x_2)^T V(x_2),
\end{equation}
where $V(x_2)$ is the (real) Vandermonde matrix
\begin{equation} \label{Vander}
V(x_2) = [\lambda_{k+1}(x_2)^{j}]_{k,j=0,\ldots,d-1}.
\end{equation}
In addition, one may easily check (e.g., using \eqref{Vander_fact}) that
\begin{equation}\label{CHHC} C(x_2)^T H(x_2) = H(x_2) C(x_2). \end{equation}

By the positive definiteness of $H(x_2)$ for all real $x_2$, we may factor $H(x_2)$ as
\begin{equation}\label{fact} H(x_2) = Q(x_2)^* Q(x_2),\quad x_2\in\mathbb{R}, \end{equation} 
where $Q(x_2)$ is a matrix polynomial of degree $d$ and $Q(x_2)$ is invertible for $\operatorname{Im} x_2 \ge 0$; 
see, for instance, \cite{RR}.
We now let 
$$M(x_2) = Q(x_2) C(x_2) Q(x_2)^{-1},$$ 
and obtain that
$$
\check p_{x_2}(t) = \det (tI_d - M(x_2)).
$$
Note that $M(x_2)=M(x_2)^*$ for $x_2 \in {\mathbb R}$. 
Indeed, \eqref{CHHC} implies that 
$$C(x_2)^*Q(x_2)^*Q(x_2) = Q(x_2)^*Q(x_2)C(x_2), \quad x_2 \in {\mathbb R}.$$ 
Multiplying on the left with $Q(x_2)^{*-1}$ and on the right with $Q(x_2)^{-1}$, 
yields that $M(x_2) = M(x_2)^*$, $x_2\in {\mathbb R}$.

Next, we claim that the rational matrix function $M(x_2)$ is in fact a matrix polynomial. 
The only possible poles arise from the zeros of $Q(x_2)$.
Let $a$ be a zero of $Q(x_2)$. Then $\operatorname{Im} a <0$.
We rewrite \eqref{fact} as 
$$
H(x_2) = Q(\bar x_2)^* Q(x_2)
$$
for all $x_2 \in {\mathbb C}$, and substitute in \eqref{CHHC}, obtaining
\begin{equation}\label{CHHC2} 
C(x_2)^T Q(\bar x_2)^* = Q(\bar x_2)^* Q(x_2) C(x_2) Q(x_2)^{-1} = Q(\bar x_2)^* M(x_2),
\end{equation}
for all $x_2\in{\mathbb C}$.
Since $Q(\bar a)$ is invertible, we conclude that
$$M(x_2) = Q(\bar x_2)^{*-1} C(x_2)^T Q(\bar x_2)^*,$$ 
is regular at $a$, i.e., $a$ is not a pole of $M(x_2)$.



It follows now from Lemma \ref{polynomial_implies_linear} that $\deg M \le 1$, i.e.,
we can write 
$$M(x_2) = -A_1 - A_2x_2,$$ 
where $A_1$ and $A_2$ are $d\times d$ Hermitian matrices. Then 
$$\check p_{x_2}(t) = \det (tI_d + A_1 + A_2x_2),$$ 
and thus
$$
p(x_1,x_2)= x_1^d \check p_{\frac{x_2}{x_1}}\Big(\frac{1}{x_1}\Big)= \det(I_d + x_1A_1+x_2A_2).
$$
\end{proof}

Note that the proof provides a constructive way to find a representation \eqref{RZP}. 
We illustrate this with an example.

\begin{ex} Let $p(x,y) = 1+10y+4x -y^2 -2xy -x^2.$ Then $\check p_y(t) = t^2 + (10y+4)t +(-1-2y-y^2).$ We get
$$H(y) = \begin{bmatrix} 2 & -10y-4\cr -10 y-4 & 102y^2+84y+18 \end{bmatrix}.$$
Factoring as in \eqref{fact} we find that
$$ Q(y) = \begin{bmatrix} \sqrt{2} &-2\sqrt{2}-5\sqrt{2}y\cr 0 & \sqrt{10}( 1+ y\frac{11+3i}{5}) \end{bmatrix}. $$ 
Then
$$ M(y) = Q(y)C(y)Q(y)^{-1} = \begin{bmatrix} -5y-2 & \sqrt{5}(1+\frac{y}{5}(11-3i)) \cr 
\sqrt{5}(1+\frac{y}{5}(11+3i)) & -5y-2 \end{bmatrix} . $$
Ultimately, we find that 
$$ p(x,y)= \det\left( I_2 + x \begin{bmatrix} 2 & -\sqrt{5} \cr -\sqrt{5} & 2 \end{bmatrix} + y \begin{bmatrix} 5 & \frac{\sqrt{5}}{5}(-11+3i)  \cr \frac{\sqrt{5}}{5}(-11-3i) & 5 \end{bmatrix}\right).$$
By the way, the polynomial $p$ was constructed using $A_1=\begin{bmatrix} 1 & 2\cr 2 & 3 \end{bmatrix}$ and $A_2=\begin{bmatrix} 4 & 5 \cr 5 & 6 \end{bmatrix}$.
\end{ex}

Before presenting a second proof of Theorem \ref{RZ}, we introduce a definition.
Let $p$ be a real-zero polynomial of total degree $d$, $p(0,0)=1$, 
and let $q$ be a real-zero polynomial of total degree less than $d$, $q(0,0) > 0$.
We define
$$\check p_{x_2}(t):= t^d p(1/t, x_2/t),\quad \check q_{x_2}(t):= t^{d-1} q(1/t, x_2/t),$$
and let, for $x_2 \in {\mathbb R}$, 
$\lambda_1(x_2) \leq \cdots \leq \lambda_d(x_2)$ and $\mu_1(x_2) \leq \cdots \leq \mu_{d-1}(x_2)$ be the zeros
of $\check p_{x_2}$ and of $\check q_{x_2}$, respectively, counting multiplicities.
We will say that $q$ {\em interlaces} $p$ if 
\begin{equation} \label{interlace}
\lambda_1(x_2) \leq \mu_1(x_2) \leq \lambda_2(x_2) \leq \cdots
\leq \lambda_{d-1}(x_2) \leq \mu_{d-1}(x_2) \leq \lambda_d(x_2)
\end{equation}
for all $x_2 \in {\mathbb R}$.
We will say that $q$ {\em strictly interlaces} $p$ if all the zeros of $\check p_{x_2}$ are simple
and strict inequalities hold in \eqref{interlace}, for all $x_2 \in {\mathbb R}$.

As an example, let $p$ be a real-zero polynomial of total degree $d$, $p(0,0)=1$, 
and let $(x^0_1,x^0_2)$ belong to the connected component of $(0,0)$
in $\{(x_1,x_2)\in{\mathbb R}^2 \colon p(x_1,x_2)>0\}$.
We set ${\mathbf p}(x_0,x_1,x_2)=x_0^d p(x_1/x_0,x_2/x_0)$ and define
$$
q(x_1,x_2) = \left.\frac{d}{ds}\,{\mathbf p}(1+s,x_1+sx^0_1,x_2+sx^0_2)\right|_{s=0};
$$
$q$ is called the Renegar derivative of $p$ with respect to $(x^0_1,x^0_2)$
and it interlaces $p$; see \cite{Re06,NPS10}.
The interlacing is strict if all the zeros of $\check p_{x_2}$ are simple
for all $x_2 \in {\mathbb R}$.
Notice that for $(x^0_1,x^0_2)=(0,0)$, we have simply
$\check q_{x_2} = \left(\check p_{x_2}\right)'$.

\begin{proof}[Second Proof of Theorem \ref{RZ}]
We assume as in the first proof that $p$ is a real-zero polynomial of total degree $d$, $p(0,0)=1$, such that
the polynomial $\check p_{x_2}$ has only simple zeros for all $x_2 \in {\mathbb R}$.
We choose a real-zero polynomial $q$ of total degree less than $d$, $q(0,0) \neq 0$,
that strictly interlaces $p$.
We let $B(x_2)$ be the {\em Bezoutian} of the polynomials $\check q_{x_2}$ and $\check p_{x_2}$, namely
(see, e.g., \cite{KN36}) 
$$B(x_2) = [b_{ij}(x_2)]_{i,j=0,\ldots,d-1},$$
where $b_{ij}(x_2)$ are determined from
\begin{equation} \label{bezoutian}
\frac{\check q_{x_2}(t) \check p_{x_2}(s) - \check q_{x_2}(s) \check p_{x_2}(t)}{t-s}
= \sum_{i,j=0,\ldots,d-1} b_{ij}(x_2) t^is^j.
\end{equation}
It is easily seen that $b_{ij}(x_2)$ are polynomials (over ${\mathbb Z}$) 
in the coefficients of $\check p_{x_2}$ and $\check q_{x_2}$,
hence polynomials in $x_2$, i.e., $B$ is a matrix polynomial.
The defining equation \eqref{bezoutian} can be conveniently rewritten as
\begin{equation} \label{Bezout-Vander1}
v_{d-1}(t) B(x_2) v_{d-1}(s)^T
= \frac{\check q_{x_2}(t) \check p_{x_2}(s) - \check q_{x_2}(s) \check p_{x_2}(t)}{t-s},
\end{equation}
and taking the limit $t \to s$,
\begin{equation} \label{Bezout-Vander2}
v_{d-1}(s) B(x_2) v_{d-1}(s)^T
= \left(\check q_{x_2}\right)'(s) \check p_{x_2}(s) - \check q_{x_2}(s) \left(\check p_{x_2}\right)'(s),
\end{equation}
where
$$
v_{d-1}(t) = \begin{bmatrix} 1 & t & \cdots & t^{d-1}\end{bmatrix}.
$$
Since the zeros of $\check p_{x_2}$ and $\check q_{x_2}$ are real, simple, and alternate for real $x_2$, 
we have that $B(x_2) > 0$, $x_2 \in {\mathbb R}$. 
This well known fact can be seen immediately by using \eqref{Bezout-Vander1}--\eqref{Bezout-Vander2} to compute
$$
V(x_2)B(x_2)V(x_2)^T = \operatorname{diag} \left(
- \check q_{x_2}(\lambda_i(x_2)) \left(\check p_{x_2}\right)'(\lambda_i(x_2)) \right)_{i=1,\ldots,d},
$$
where $V(x_2)$ is the Vandermonde matrix \eqref{Vander} based at the zeros of $\check p_{x_2}$.
In addition, one may easily check 
\begin{equation}\label{CHHC-Bezout} C(x_2) B(x_2) = B(x_2) C(x_2)^T \end{equation}
--- e.g., multiplying both sides by $V(x_2)$ from the left and $V(x_2)^T$ from the right,
and using \eqref{Bezout-Vander1}--\eqref{Bezout-Vander2}.

By the positive definiteness of $B(x_2)$ for all real $x_2$, we may factor $B(x_2)$ as
\begin{equation}\label{fact-Bezout} B(x_2) = P(x_2) P(x_2)^*,\quad x_2\in\mathbb{R}, \end{equation} 
where $P(x_2)$ is a matrix polynomial and $P(x_2)$ is invertible for $\operatorname{Im} x_2 \ge 0$,
and we let 
$$M(x_2) = P(x_2)^{-1} C(x_2) P(x_2),$$ 
and obtain that
$$
\check p_{x_2}(t) = \det (tI_d - M(x_2)).
$$
As in the first proof of the theorem,
\eqref{CHHC-Bezout} and \eqref{fact-Bezout} imply that 
$M(x_2) = M(x_2)^*$, $x_2\in {\mathbb R}$,
and that the rational matrix function $M(x_2)$ is regular at a zero $a$ of $P(x_2)$,
so that it is in fact a matrix polynomial\footnote{
Alternatively, we can prove that a zero $a$ of $P(x_2)$ is not a pole of $M(x_2)$ similarly to the proof of the claim
in the proof of Theorem \ref{main}.
It is well known that $\det B(a) = 0$ iff the polynomials $\check p_{a}$ and $\check q_{a}$ have a common
zero $\lambda$; let us assume that $\lambda$ is a simple zero of both $\check p_{a}$ and $\check q_{a}$,
then it is also well known that the left kernel of $B(a)$ is spanned by $v_{d-1}(\lambda)$
(all these facts follow quite easily from \eqref{Bezout-Vander1}--\eqref{Bezout-Vander2}).
Since $B(a) = P(a) P(\bar a)^*$, and since $v_{d-1}(\lambda) C(a) = \lambda v_{d-1}(\lambda)$,
it follows that the one-dimensional left kernel of $P(a)$ is the left eigenspace of $C(a)$,
implying as in the proof of Theorem \ref{main} that $a$ is not a pole of $M(x_2)$.}.
It follows from Lemma \ref{polynomial_implies_linear} that $M$ is linear:
$$M(x_2) = -A_1 - A_2x_2,$$ 
where $A_1$ and $A_2$ are $d\times d$ Hermitian matrices, and then 
$$
p(x_1,x_2)= x_1^d \check p_{\frac{x_2}{x_1}}\Big(\frac{1}{x_1}\Big)= \det(I_d + x_1A_1+x_2A_2).
$$
\end{proof}

This second proof of Theorem \ref{RZ} is of course constructive as well as soon as we choose a strictly interlacing
polynomial $q$.

We notice also that the algebro-geometrical proof of Theorem \ref{RZ} given in \cite{Vppf} and in \cite{PV} also
uses an interlacing polynomial $q$, and yields a determinantal representation with
$$
q(x_1,x_2) = c \operatorname{adj}(I+x_1A_1+x_2A_2) \, c^T,
$$
where $\operatorname{adj}$ denotes the classical adjoint or adjugate matrix (the matrix of cofactors) and $c \in {\mathbb C}^{1 \times d}$.
It would be interesting to see whether this relation holds for the determinantal representation constructed 
in the second proof of Theorem \ref{RZ} above (meaning that the two constructions are essentially equivalent,
despite using quite different methods).

\begin{rem} 
Note that for $d=2$ we can always convert a representation $p(x_1,x_2)=\det (I_2+ x_1A_1 + x_2A_2)$ 
with $A_1$ and $A_2$ Hermitian, 
to one with real symmetric $A_1$ and $A_2$. 
Indeed, write $A_1=UDU^*$, with $U$ unitary and $D$ diagonal, 
and consider $U^*A_2U$ which has a complex $(1,2)$ entry with, say, argument $\theta$. 
Then letting $V=\small \begin{bmatrix} 1 & 0 \cr 0 & e^{i\theta} \end{bmatrix}$ 
and $\hat A_1 = D = VDV^*, \hat A_2 = VUA_2U^*V^* \in {\mathbb R}^{2\times 2}$, 
we obtain  $p(x_1,x_2)=\det (I_2+ x\hat A_1 + x_2\hat A_2)$, as desired.
\end{rem}

\begin{rem}\label{rem:p-square} (See \cite[Section 1.4]{RG95} and \cite[Lemma 2.14]{NT1}.)
From the representation as in Theorem \ref{RZ}, we may represent $p(x_1,x_2)^2$ as
\begin{equation}\label{pxy2} p(x_1,x_2)^2=\det(I_{2d} + x_1 \alpha_1 + x_2 \alpha_2 ), \end{equation}
where $\alpha_1=\alpha_1^T, \alpha_2=\alpha_2^T \in {\mathbb R}^{2d \times 2d}$. 
Indeed, with $A_1$ and $A_2$ as in Theorem \ref{RZ}, we write
$$ A_1=A_{1R} + i A_{1I} , A_2 = A_{2R} + i A_{2I}, $$
where $A_{1R}, A_{2R}, A_{1I}, A_{2I} \in {\mathbb R}^{d\times d}.$ 
It is easy to check that since $A_1$ and $A_2$ are Hermitian, $A_{1R}$, $A_{2R}$ are symmetric 
and $A_{1I}$, $A_{2I}$ are skew-symmetric. 
Let now 
$$\alpha_1 = \begin{bmatrix} A_{1R} & A_{1I}\cr -A_{1I} & A_{1R} \end{bmatrix},\qquad 
\alpha_2 = \begin{bmatrix} A_{2R} & A_{2I}\cr -A_{2I} & A_{2R} \end{bmatrix}, $$
and \eqref{pxy2} follows. Indeed, using 
$$U=\frac{1}{\sqrt{2}}\begin{bmatrix} I_d & I_d \cr iI_d & -i I_d \end{bmatrix}, $$ 
it is easy to check that
$$ U \begin{bmatrix} A_1 & 0 \cr 0 & A_1^T \end{bmatrix} U^* = \alpha_1,\qquad 
U \begin{bmatrix} A_2 & 0 \cr 0 & A_2^T \end{bmatrix} U^* = \alpha_2. $$
\end{rem}

\end{document}